\documentclass[11pt, reqno]{amsart}
\usepackage{amssymb, amsthm, amsmath, amsfonts,mathtools}
\usepackage{array, epsfig}
\usepackage{bbm}
\usepackage{enumitem}

\usepackage{hyperref}
\usepackage[numbers,square]{natbib}
\usepackage{color}
\allowdisplaybreaks

\numberwithin{equation}{section}

  \evensidemargin
\oddsidemargin
\newcommand{\stirling}[2]{\genfrac{[}{]}{0pt}{}{#1}{#2}}

\newcommand{\stirlingb}[2]{B\hspace*{-0.2mm}\big[{#1},{#2}\big]}

\newcommand{\stirlingd}[2]{D\hspace*{-0.2mm}\big[{#1},{#2}\big]}

\newcommand{\N}{\mathbb{N}}

\newcommand{\R}{\mathbb{R}}
\newcommand{\C}{\mathbb{C}}

\newcommand{\F}{\mathcal{F}}

\newcommand{\PP}{\mathbb{P}}

\newcommand{\eps}{\varepsilon}

\newcommand*\xbar[1]{%
   \hbox{%
     \vbox{%
       \hrule height 0.5pt 
       \kern0.25ex
       \hbox{%
         \kern-0.05em
         \ensuremath{#1}%
         \kern-0.1em
       }%
     }%
   }%
}

\DeclareMathOperator{\Bern}{Bern}

\DeclareMathOperator{\lin}{lin}

\DeclareMathOperator{\pos}{pos}

\DeclareMathOperator{\relint}{relint}

\renewcommand{\P}{\mathbb{P}}

\theoremstyle{plain}
\newtheorem{satz}{Theorem}[section]
\newtheorem{lem}[satz]{Lemma}

\newtheorem{prop}[satz]{Proposition}

\theoremstyle{definition}

\theoremstyle{remark}

\begin{document}

\author{Thomas Godland$^*$}
\address{Institut f\"ur Mathematische Stochastik,
Westf\"alische Wilhelms-Universit\"at M\"unster,
Orl\'eans-Ring 10,
48149 M\"unster, Germany}
\email[$^*$Corresponding author]{t\_godl01@uni-muenster.de}

\author{Zakhar Kabluchko}
\address{Institut f\"ur Mathematische Stochastik,
Westf\"alische Wilhelms-Universit\"at M\"unster,
Orl\'eans-Ring 10,
48149 M\"unster, Germany}
\email{zakhar.kabluchko@uni-muenster.de}

\title[conic intrinsic volumes of Weyl chambers]{conic intrinsic volumes of Weyl chambers}



\begin{abstract}
We give a new, direct  proof of the formulas for the conic intrinsic volumes of the Weyl chambers of types $A_{n-1}, B_n$ and $D_n$. These formulas express the conic intrinsic volumes in terms of the Stirling numbers of the first kind and their $B$- and $D$-analogues.  The proof involves an explicit determination of the internal and external angles of the faces of the Weyl chambers.
\end{abstract}

\maketitle

\section{Introduction}
A polyhedral cone in the Euclidean space $\R^n$ is a set of solutions to a finite system of linear homogeneous  inequalities. That is, a polyhedral cone $C\subseteq \R^n$ can be represented as
$$
C = \{\beta\in \R^n: \langle \beta , x_i\rangle \leq 0 \text{ for all } i=1,\ldots,m\}
$$
for some finite collection of vectors $x_1,\ldots,x_m\in \R^n$, where $\langle \cdot, \cdot \rangle$ denotes the standard Euclidean scalar product.
The fundamental Weyl chambers of types $A_{n-1}, B_n$ and $D_n$  are the polyhedral cones defined by
\begin{align*}
\mathcal{C}(A_{n-1})&:=\{\beta\in\R^n:\beta_1\ge\beta_2\ge\ldots\ge \beta_n\},\\
\mathcal{C}(B_n)&:=\{\beta\in\R^n:\beta_1\ge\beta_2\ge\ldots\ge \beta_n\ge 0\},\\
\mathcal{C}(D_n)&:=\{\beta\in\R^n:\beta_1\ge\beta_2\ge \ldots\ge \beta_{n-1} \ge |\beta_n|\},
\end{align*}
where $\beta = (\beta_1,\ldots, \beta_n)$ is the coordinate representation of $\beta\in\R^n$.

In this paper, we shall be interested in the conic intrinsic volumes of the Weyl chambers. The conic intrinsic volumes of cones are analogues of the classical Euclidean intrinsic volumes of convex bodies in the setting of conic or spherical geometry. A recent increase of interest to conic intrinsic volumes is due to their relevance in convex optimization~\cite{Amelunxen2014,amelunxen_buergisser,amelunxen_buergisser_grass}.  Let us briefly define the conic intrinsic volumes, referring to Section~\ref{subsec:conic} for more details and to~\cite[Section 6.5]{Schneider2008} and~\cite{Amelunxen2014,Amelunxen2017} for an extensive account of the theory.
Given some point $x\in\R^n$, the Euclidean projection $\Pi_C(x)$ of $x$ to a polyhedral cone $C\subseteq \R^n$ is the unique vector $y\in C$ minimizing the Euclidean distance $\|x-y\|$.
For $k\in \{0,\ldots,n\}$, the $k$-th conic intrinsic volume $\upsilon_k(C)$ of $C$ is defined as the probability that the Euclidean projection $\Pi_C(g)$ of an $n$-dimensional standard Gaussian random vector $g$ on $C$ lies in the relative interior of a $k$-dimensional face of $C$ (meaning that the projection is contained  in this face but not in any face of smaller dimension).
The conic intrinsic volumes of the Weyl chambers are given by the following theorem.
\begin{satz}\label{satz_intro}
For all $n\in \{1,2,\ldots\}$ and all $k\in \{0,\ldots,n\}$ we have
\begin{align}\label{Eq_Intr_Vol_Weyl_Chamb}
\upsilon_k(\mathcal{C}(A_{n-1}))=\stirling{n}{k}\frac{1}{n!},
\quad
\upsilon_k(\mathcal{C}(B_n))=\frac{\stirlingb{n}{k}}{2^nn!}
\quad\text{and}\quad
\upsilon_k(\mathcal{C}(D_n))=\frac{D[n,k]}{2^{n-1}n!},
\end{align}
where the $\stirling{n}{k}$'s denote the Stirling numbers of the first kind, the $\stirlingb{n}{k}$'s their $B$-analogues and the $\stirlingd {n}{k}$'s their $D$-analogues defined as the coefficients of the following polynomials:
\begin{align*}
t(t+1)\cdot\ldots\cdot(t+n-1)&=\sum_{k=0}^n\stirling{n}{k}t^k,\\
(t+1)(t+3)\cdot\ldots\cdot(t+2n-1)&=\sum_{k=0}^n\stirlingb{n}{k}t^k,\\
(t+1)(t+3)\cdot\ldots\cdot(t+2n-3)(t+n-1)&=\sum_{k=0}^n\stirlingd {n}{k}t^k.
\end{align*}
\end{satz}

The theorem is known and appeared in various, seemingly unrelated  forms. One of the approaches to prove it is to relate the conic intrinsic volumes to the coefficients of the characteristic polynomial of the hyperplane arrangement generating the Weyl chambers; see~\cite[Theorems~3,4]{drton_klivans}, \cite[Theorem~5]{klivans_swartz}, \cite[Section~6]{AmelunxenLotzDCG17}, \cite{schneider_combinatorial}, \cite[Theorems~4.1 and~4.2]{KVZ15}.  The characteristic polynomials of Weyl arrangements can be computed by the finite field method; see Section 5.1 in~\cite{stanley_book} or Section~1.7.4 in~\cite{bona_handbook}.


On the other hand, the $A_{n-1}$-case of the formula~\eqref{Eq_Intr_Vol_Weyl_Chamb} is closely related to a result,  due to Sparre Andersen~\cite[Theorem~5]{Sparre},  on the least concave majorant of a random walk. Consider a random walk given by $S_k := g_1+\ldots+g_k$, $1\leq k\leq n$, and $S_0:=0$, where the increments $g_1,\ldots,g_n$ are independent standard Gaussian random variables. (In fact, the result stated below holds, essentially, under the assumption of exchangeability of the increments). The graph of the random walk is the set in $\R^2$ consisting of the points $(k,S_k)$, for $k\in \{0,1,\ldots,n\}$. The least concave majorant of the random walk is the smallest concave function defined on $[0,n]$ that stays above this graph. This function is piecewise linear. If $N_n$ denotes the number of linearity segments of the least concave majorant, then a formula due to Sparre Andersen~\cite[Theorem~5]{Sparre} states that
\begin{equation}\label{eq:sparre}
\P[N_n = k] = \stirling{n}{k}\frac{1}{n!}, \qquad k\in \{1,\ldots,n\}.
\end{equation}
For further development of this probabilistic result we refer to~\cite{Abramson+Pitman:2011,Abramson+Pitman+Ross+Bravo:2011}. Now, the projection of the point $g=(g_1,\ldots,g_n)$ on the Weyl chamber $\mathcal{C}(A_{n-1})$ can explicitly be described in terms of the least concave majorant of the corresponding random walk, and, in particular, it is possible to check that the number of linearity segments of the majorant coincides with the dimension of the unique face $F$ such that $\Pi_{\mathcal{C}(A_{n-1})}(g)$ lies in the relative interior of $F$. With this knowledge, the $A_{n-1}$-case of the formula~\eqref{Eq_Intr_Vol_Weyl_Chamb} can be viewed as a consequence of Sparre Andersen's formula~\eqref{eq:sparre}.


In the present paper, we shall give yet another proof of Theorem~\ref{satz_intro}. The starting point of this proof is a formula, see~\eqref{Eq_In_Vol_InternalExternal} below, expressing the conic intrinsic volumes of a polyhedral cone in terms of the internal and external angles of its  faces.   Probably the main advantage of our proof is an explicit identification of all internal and external angles involved in this formula.
The determination of the external angles is closely related to a result of Gao and Vitale~\cite{Gao2001} who computed the classical intrinsic volumes of the Schl\"afli orthoscheme which is defined as the simplex in $\R^n$ with the vertices $0$ and $e_1+\ldots+e_i$, $1\leq i\leq n$, where $e_1,\ldots,e_n$ denotes the standard orthonormal basis in $\R^n$.
Equivalently, this simplex is given by
\begin{equation}\label{eq:gao_vitale_simplex}
\{\beta\in\R^n:1\ge \beta_1\ge\beta_2\ge\ldots\ge \beta_n\geq 0\}.
\end{equation}
Later, Gao~\cite{gao} computed the intrinsic volumes of the simplex with the vertices  $e_1+\ldots+e_i - \frac in (e_1+\ldots+e_n)$, $1\leq i \leq n$, which can be also given by
$$
\{\beta\in\R^n: \beta_1\ge\beta_2\ge\ldots\ge \beta_n, \beta_1+\ldots+\beta_n = 0, \beta_1-\beta_n \leq 1\}.
$$
These simplices are closely related to the Weyl chambers of types $B_n$ and $A_{n-1}$, respectively. Both the computation of the external angles of the faces of the Weyl chambers (which follows the method of Gao and Vitale~\cite{Gao2001} and Gao~\cite{gao}), and the computation of the  internal angle sums proceed by re-arranging the solid angles under interest in such a way that they cover the whole space, from which we conclude that the sum of the angles is $1$.

The rest of the paper is mostly devoted to the proof of Theorem~\ref{satz_intro}.

\section{Preliminaries}

\subsection{Conic intrinsic volumes and solid angles} \label{subsec:conic}

In this section we collect some information on polyhedral cones (called just cones, for simplicity) that is needed in later proofs. A supporting hyperplane for a cone $C\subseteq \R^n$ is a linear hyperplane $H$ with the property that  $C$ lies entirely in one of the closed half-spaces $H^+$ and $H^-$ induced by $H$. A \textit{face} of $C$ is a set of the form $F=C\cap H$, for a supporting hyperplane $H$, or the cone $C$ itself. We denote by $\F_k(C)$ the set  of all  $k$-dimensional faces of $C$, for $k\in \{0,\ldots,n\}$. The dimension of a face $F$ is defined as the dimension of its linear hull, i.e.~$\dim F=\dim\lin(F)$. Equivalently, the faces of $C$ are obtained by replacing some of the half-spaces whose intersection defines the polyhedral cone by their boundaries and taking the intersection.

The \textit{dual cone} (or the \textit{polar cone}) of a cone $C\subseteq \R^n$ is defined as
\begin{align*}
C^\circ=\{x\in\R^n:\langle x,y\rangle\le 0\; \text{ for all } y\in C\}.
\end{align*}
For example, if $C=L$ is a linear subspace, then $C^\circ=L^\perp$ is its orthogonal complement. There is a one-to-one correspondence between the $k$-faces of $C$ and the $(n-k)$-faces of $C^\circ$ via the bijective mapping
\begin{align*}
\begin{array}{ccc}
\F_k(C) & \to & \F_{n-k}(C^\circ) \\
F & \mapsto & N(F,C)
\end{array},
\end{align*}
where $N(F,C):=(\lin F)^\perp\cap C^\circ$ is called the \textit{normal face} (or normal cone) of $F$ with respect to $C$.

The \textit{positive hull} of a finite set $\{x_1,\ldots,x_m\}\subseteq \R^n$ is defined as the smallest cone containing this set, that is
$$
\pos \{x_1,\dots,x_m\} = \{\lambda_1 x_1 + \ldots + \lambda_m x_m : \lambda_1\geq 0, \ldots, \lambda_m \ge 0\}.
$$
We will repeatedly make use of the well-known duality relations
\begin{align}\label{Eq_Duality_Pos_Cap}
\pos\{x_1,\dots,x_m\}^\circ=\bigcap_{i=1}^mx_i^-\quad\text{and}\quad\pos\{x_1,\dots,x_m\}=\bigg(\bigcap_{i=1}^mx_i^-\bigg)^\circ
\end{align}
for $x_1,\dots,x_m\in\R^n$ and $x_i^-:=\{v\in\R^n:\langle v,x_i\rangle\le 0\}$, $i=1,\dots,m$.

Now, we define the conic intrinsic volumes. The definition and further properties are taken from \cite[Section 2.2]{Amelunxen2017} and \cite[Section 2]{HugSchneider2016}.

Let $C\subseteq \R^n$ be a polyhedral cone, and $g$ be an $n$-dimensional standard Gaussian random vector. Then, for $k\in \{0,\ldots,n\}$, the $k$-th \textit{conic intrinsic volume} (or, for simplicity, just \textit{intrinsic volume}) of $C$ is defined by
\begin{align*}
\upsilon_k(C):=\sum_{F\in\F_k(C)}\PP(\Pi_C(g)\in\relint(F)),
\end{align*}
Here, $\Pi_C$ denotes the orthogonal projection on $C$, that is $\Pi_C(x)$ is the vector in $C$ minimizing the Euclidean distance to $x\in\R^n$. Also, $\relint(F)$ denotes the interior of $F$ taken with respect to its linear hull $\lin (F)$ as an ambient space.

The \textit{solid angle} (or just angle) of a cone $C\subseteq \R^n$, denoted by $\alpha(C)$, is defined  as
\begin{align*}
\alpha(C)=\P(Z\in C),
\end{align*}
where $Z$ is uniformly distributed on the unit sphere in the linear hull $\lin C$ of $C$. In fact, for the distribution of $Z$ we can choose any rotation invariant distribution living on the linear hull of $C$.
%
%
%
%
For a $k$-dimensional cone $C\subseteq\R^n$, $k\in\{1,\dots,n\}$, the $k$-th conic intrinsic volume coincides with the solid angle of $C$, that is
$$
\upsilon_k(C) = \alpha (C).
$$

The \textit{external angle} of a cone $C\subseteq\R^n$ at a face $F$ is defined as the solid angle of the normal face of $F$ with respect to $C$, that is the external angle is $\alpha(N(F,C))$.
The \textit{internal angle} of a face $F$ at $0$ is defined as $\alpha(F)$.
For the purpose of this paper, we need the well-known expression of the intrinsic volumes in terms of internal and external angles:
\begin{align}\label{Eq_In_Vol_InternalExternal}
\upsilon_k(C)
&	
=
\sum_{F\in\F_k(C)}\alpha(F)\alpha(N(F,C)),
\qquad k\in \{0,\ldots,n\}.
\end{align}
We shall compute the intrinsic volumes of Weyl chambers by evaluating the internal and external angles at their faces.

\subsection{Stirling numbers of the first kind and their generating functions}\label{Section_Stirling_numbers}
In this section we recall some facts on the Stirling numbers and their $B$- and $D$-analogues. These numbers, well known in combinatorics, appear in the formulas for the intrinsic volumes of the Weyl chambers.
The (unsigned) Stirling numbers of the first kind are denoted by $\stirling{n}{k}$ and defined as the coefficients of the polynomial
\begin{align}\label{Eq_Def_Stirling1}
t(t+1)\cdot\ldots\cdot(t+n-1)=\sum_{k=1}^n\stirling{n}{k}t^k.
\end{align}
By convention, $\stirling{n}{k}=0$ for $n\in \N$, $k\notin\{1,\dots,n\}$.  Equivalently, $\stirling{n}{k}$ can be defined as the number of permutations of  the set $\{1,\dots,n\}$ having exactly $k$ cycles. Other representations of the Stirling numbers of the first kind are known, e.g.
\begin{align}\label{Eq_Stirling1_AsComp}
\stirling{n}{k}=\frac{n!}{k!}\sum_{\substack{i_1,\dots,i_k\in\N\\i_1+\ldots+i_k=n}}\frac{1}{i_1i_2\cdot \ldots\cdot i_k};
\end{align}
see \cite[(1.9) and (1.13)]{Pitman2006}. We shall also need the generating functions of the Stirling numbers of the first kind:
\begin{align}
\sum_{n=0}^\infty  \stirling{n}{k} \frac{t^n}{n!}=\frac{(-\log(1-t))^k}{k!}\quad\text{and}\quad    \sum_{n=0}^\infty\sum_{k=0}^n\stirling{n}{k}\frac{t^n}{n!}y^k=(1-t)^{-y},\label{Eq_Gen_Stirling1A}
\end{align}
for all complex $t$ such that $|t|<1$ and all $y\in\C$. By convention,  $\stirling{0}{0} = 1$.

The $B$-analogues of the (signless) Stirling numbers of the first kind are denoted by $\stirlingb{n}{k}$ and defined as the coefficients of the polynomial
\begin{align}\label{Eq_Def_Stirling1B}
(t+1)(t+3)\cdot\ldots\cdot(t+2n-1)=\sum_{k=0}^n\stirlingb{n}{k}t^k.
\end{align}
Again, by convention, we put $\stirlingb{n}{k}=0$ for $k\notin\{0,\dots,n\}$. These numbers appear as entry A028338 (or A039757 for the signed version) in the On-Line Encyclopedia of Integer Sequences~\cite{sloane} and also in~\cite{Amelunxen2017,bagno_biagioli_garber_some_identities,bagno_garber_balls,bala_stirling,dowling,drton_klivans,
henze_orakel,KVZ17,KVZ15,lang_stirling,suter}.
As follows easily from their definition,  the numbers $\stirlingb nk$ satisfy the recurrence relation
\begin{align}\label{Eq_Recurrence_B(n,k)}
\stirlingb {n}{k}=(2n-1)\stirlingb {n-1}{k}+\stirlingb {n-1}{k-1};
\end{align}
see \cite[Section 2.2]{KVZ15}.
The following proposition collects some useful explicit formulas for $B[n,k]$ including their exponential generating function. 

\begin{prop}\label{prop:formulas_B(n,k)}
The $B$-analogues $\stirlingb{n}{k}$ of the Stirling numbers of the first kind are explicitly given by the formulas
\begin{align}\label{Eq_B(n,k)_expl}
&\stirlingb{n}{k}=\sum_{i=k}^n2^{n-i}\binom{i}{k}\stirling{n}{i},\\
&\stirlingb{n}{k}=\sum_{r=0}^{n-k}2^{n-k-2r}\binom{2r}{r}\stirling{n-r}{k}\frac{n!}{(n-r)!},\label{Eq_Equality_Generating_functions}
\end{align}
for $k\in\{0,\dots,n\}$. The exponential generating function of the array $(\stirlingb {n}{k})_{n,k\ge 0}$ is given by
\begin{align}\label{eq:generating_fct_B}
\sum_{n=0}^{\infty}\sum_{k=0}^{n}\stirlingb {n}{k}\frac{t^n}{n!}y^k=(1-2t)^{-\frac{1}{2}(y+1)}
\end{align}
for all complex $|t|<1/2$ and $y\in \C$. By convention, $\stirlingb{0}{0} = 1$. 
\end{prop}

In Entry A028338 of~\cite{sloane} the explicit formula~\eqref{Eq_B(n,k)_expl} for the number $\stirlingb{n}{k}$ in terms of the Stirling numbers of the first kind was stated by F.\ Woodhouse without a proof.

\begin{proof}[Proof of Proposition~\ref{prop:formulas_B(n,k)}]
We start with the proof of~\eqref{Eq_B(n,k)_expl}. We want to check whether the numbers on the right-hand side of~\eqref{Eq_B(n,k)_expl} coincide with the coefficients of the polynomial in~\eqref{Eq_Def_Stirling1B}. We have
\begin{align*}
\sum_{k=0}^n\sum_{i=k}^n2^{n-i}\binom{i}{k}\stirling{n}{i}t^k
&	=2^n\sum_{i=0}^n2^{-i}\stirling{n}{i}\bigg(\sum_{k=0}^i\binom{i}{k}t^k\bigg)	=2^n\sum_{i=0}^n\stirling{n}{i}\Big(\frac{t+1}{2}\Big)^i\\
&	=2^n\Big(\frac{t+1}{2}\Big)\Big(\frac{t+1}{2}+1\Big)\cdot\ldots\cdot\Big(\frac{t+1}{2}+n-1\Big)\\
&	=(t+1)(t+3)\cdot\ldots\cdot(t+2n-1)
\end{align*}
using the Binomial Theorem and~\eqref{Eq_Def_Stirling1}, which completes the proof of~\eqref{Eq_B(n,k)_expl}.

Next, we prove~\eqref{eq:generating_fct_B}.  By~\eqref{Eq_Def_Stirling1B} and~\eqref{Eq_Def_Stirling1} we have
\begin{align*}
\sum_{k=0}^n \stirlingb {n}{k}y^k
=
2^n\Big(\frac{y+1}{2}\Big)\Big(\frac{y+1}{2}+1\Big)\cdot\ldots\cdot\Big(\frac{y+1}{2}+n-1\Big)
=\sum_{i=0}^n 2^n\stirling{n}{i}\Big(\frac{y+1}{2}\Big)^i.
\end{align*}
Thus, the generating function is given by
\begin{align*}
\sum_{n=0}^\infty\sum_{k=0}^{n}\stirlingb {n}{k}\frac{t^n}{n!} y^k=\sum_{n=0}^\infty\sum_{i=0}^n\stirling{n}{i}\Big(\frac{y+1}{2}\Big)^i\frac{(2t)^n}{n!}=(1-2t)^{-\frac{1}{2}(y+1)},
\end{align*}
where we used \eqref{Eq_Gen_Stirling1A} in the last step.

At last, we use both formulas~\eqref{Eq_B(n,k)_expl} and~\eqref{eq:generating_fct_B} to prove~\eqref{Eq_Equality_Generating_functions}. To this end, we compare the generating functions of the sequences on both sides of~\eqref{Eq_Equality_Generating_functions}.
Consider the following sum:
\begin{align*}
a_k(t):=\sum_{n=k}^\infty\left(\sum_{r=0}^{n-k}2^{n-k-2r}\binom{2r}{r}\stirling{n-r}{k}\frac{n!}{(n-r)!}\right)\frac{t^n}{n!},
\qquad k\in \N_0. 
\end{align*}
Then, we get
\begin{align}
a_k(t)\label{Eq_Gen_FUnction_Bn_Sum1}
&	=\sum_{n=k}^\infty\sum_{r=0}^{n-k}2^{n-k-2r}\binom{2r}{r}\stirling{n-r}{k}\frac{t^n}{(n-r)!}=\sum_{r=0}^{\infty}\sum_{n=r+k}^\infty 2^{n-k-2r}\binom{2r}{r}\stirling{n-r}{k}\frac{t^n}{(n-r)!}\notag\\
&	=\sum_{r=0}^{\infty}2^{-r-k}\binom{2r}{r}t^r\left(\,\sum_{n=r+k}^\infty\stirling{n-r}{k}\frac{(2t)^{n-r}}{(n-r)!}\right).
\end{align}
By shifting the index in the inner sum and using the generating function of the Stirling numbers of the first kind stated in~\eqref{Eq_Gen_Stirling1A}, we can rewrite~\eqref{Eq_Gen_FUnction_Bn_Sum1} to obtain
\begin{align*}
a_k(t) &= \sum_{r=0}^{\infty}2^{-r-k}\binom{2r}{r}t^r\left(\,\sum_{n=k}^\infty\stirling{n}{k}\frac{(2t)^{n}}{n!}\right)
	=\sum_{r=0}^{\infty}2^{-r-k}\binom{2r}{r}t^r\frac{(-\log(1-2t))^k}{k!}\\
&	=2^{-k}\frac{(-\log(1-2t))^k}{k!}\sum_{r=0}^\infty \Big(\frac{t}{2}\Big)^r\binom{2r}{r}	=2^{-k}\frac{(-\log(1-2t))^k}{k!}\frac{1}{\sqrt{1-2t}}
\end{align*}
for $|t|<1/2$. Thus, the generating function of the right-hand side of~\eqref{Eq_Equality_Generating_functions} is given by
\begin{align*}
\sum_{k=0}^\infty a_k(t)y^k
	=\frac{1}{\sqrt{1-2t}}\exp\Big(-\frac{y}{2}\log(1-2t)\Big)
	=(1-2t)^{-\frac{1}{2}(y+1)},
\end{align*}
which coincides with the generating function of $(\stirlingb{n}{k})_{n,k\ge 0}$. This completes the proof.
\end{proof}

The $D$-analogues of the (signless) Stirling numbers of the first kind are denoted by $\stirlingd {n}{k}$ and defined as the coefficients of the polynomial
\begin{align}\label{eq:D_n_k}
(t+1)(t+3)\cdot\ldots\cdot(t+2n-3)(t+n-1)=\sum_{k=0}^n\stirlingd {n}{k}t^k.
\end{align}
By convention, $\stirlingd {n}{k}=0$ for $k\notin\{0,\dots,n\}$. The signed version of these numbers appears as entry A039762 in \cite{sloane}. The $D$-analogues can be expressed through the $B$-analogues by
\begin{align}\label{Eq_Relation_D(n,k)_B(n,k)}
\stirlingd {n}{k}=(n-1)\stirlingb {n-1}{k}+\stirlingb {n-1}{k-1} =\stirlingb{n}{k}-n\stirlingb{n-1}{k};
\end{align}
see \cite[Section 2.2]{KVZ15} and use~\eqref{Eq_Recurrence_B(n,k)}.

In  probabilistic terms, \eqref{Eq_Def_Stirling1}, \eqref{Eq_Def_Stirling1B} and~\eqref{eq:D_n_k} state that
$$
\P(N_n = k) =\stirling{n}{k}\frac{1}{n!},
\quad
\P(N_n^{B} = k) = \frac{\stirlingb{n}{k}}{2^nn!},
\quad
\P(N_n^{D} = k) = \frac{\stirlingd {n}{k}}{2^{n-1}n!},
$$
for all $k\in \{0,\ldots,n\}$, where $N_n, N_n^{B}$ and $N_n^{D}$ are random variables defined by
$$
N_n = \sum_{j=1}^n \text{Bern}\left(\frac 1j\right),
\quad
N_n^{B} = \sum_{j=1}^n \text{Bern}\left(\frac 1{2j}\right),
\quad
N_n^{D} = \sum_{j=1}^{n-1} \text{Bern}\left(\frac 1{2j}\right) + \text{Bern}\left(\frac 1{n}\right),
$$
and $\Bern(\cdot)$ are independent Bernoulli variables with corresponding parameters. It is well known that $N_n$ has the same distribution as the number of cycles in a uniform random permutation of $n$ elements. This can be deduced from the Feller coupling or from the Chinese restaurant construction of the uniform random permutation, see~\cite[Section~3.1]{Pitman2006}. Similarly, it can be shown that $N_n^{B}$ has the same distribution as the number \textit{even} cycles in a uniform random \textit{signed} permutation of $n$ elements, where a signed permutation is a pair $(\sigma,\eps)$ consisting of a permutation $\sigma$ acting on $\{1,\ldots,n\}$ together with a vector of signs $\eps = (\eps_1,\ldots, \eps_n)\in \{-1,+1\}^n$, and a cycle $i_1  \mapsto \ldots \mapsto i_k \mapsto i_1$ of $\sigma$ is called even if $\eps_{i_1} \cdot \ldots \cdot \eps_{i_k} = +1$. Observe that this gives a combinatorial interpretation of~\eqref{Eq_B(n,k)_expl}.
More generally, it is known from~\cite[5.3]{shephard_todd} or~\cite[p.~59, p.~63]{humphreys_book}  that the numbers $\stirling{n}{k}$, $\stirlingb{n}{k}$ and $D[n,k]$ count  the elements in the reflection group of the corresponding type whose space of invariant vectors is $k$-dimensional.
For example, $N_n^{D}$ has the same distribution as the number of even cycles in a random uniform $D$-permutation of $n$ elements, where a $D$-permutation is a signed permutation $(\sigma,\eps)$ such that $\eps_1 \cdot \ldots \cdot \eps_n = 1$.

\section{Conic intrinsic volumes of Weyl chambers}
\label{Section_Intr_Vol_Weyl_chambers}

This section is dedicated to proving the formulas for the intrinsic volumes of the Weyl chambers stated in~\eqref{Eq_Intr_Vol_Weyl_Chamb}. To this end, we want to use the formula~\eqref{Eq_In_Vol_InternalExternal} and evaluate the internal and external angles of the faces of the Weyl chambers. The computation of the external angles relies on the ideas of Gao and Vitale~\cite{Gao2001} who derived the classical intrinsic volumes of the simplex~\eqref{eq:gao_vitale_simplex}.

\subsection{Type \texorpdfstring{$\boldsymbol{A_{n-1}}$}{A\_\{n-1\}}} We start with the simpler $A_{n-1}$ case. Recall the definition of the fundamental Weyl chamber of type $A_{n-1}$:
\begin{align*}
C^A:=\mathcal{C}(A_{n-1})=\{\beta\in\R^n:\beta_1\ge\beta_2\ge\ldots\ge\beta_n\}.
\end{align*}
Then our first result is the following.

\begin{satz}\label{Theorem_Intr_Vol_An-1}
For $k=1,\dots,n$ the $k$-th conic intrinsic volume of the Weyl chamber of type $A_{n-1}$ is given by
\begin{align*}
\upsilon_k(\mathcal{C}(A_{n-1}))=\stirling{n}{k}\frac{1}{n!},
\end{align*}
where the Stirling numbers $\stirling{n}{k}$ are defined as in Section~\ref{Section_Stirling_numbers}.
\end{satz}

Before we start with the proof of Theorem~\ref{Theorem_Intr_Vol_An-1}, we compute the solid angle of the cone whose points correspond to ``bridges staying below zero''. 

\begin{lem}\label{lemma:angle_Dm}
The solid angle of the cone
\begin{align*}
D_m
	:=&\{x\in\R^m:x_1\le 0,x_1+x_2\le 0,\dots,x_1+\ldots+x_{m-1}\le 0,x_1+\ldots+x_m=0\}
\end{align*}
is given by $\alpha(D_m)=\frac 1m$, for $m\in\N$.
\end{lem}

\begin{proof}[Proof of Lemma~\ref{lemma:angle_Dm}]
The idea is to define $m$ cones in $\R^m$ with the same solid angle as $D_m$, such that their union equals the linear subspace $H_m:=\{x\in\R^m:x_1+\ldots+x_m=0\}$ and their pairwise intersections are Lebesgue null sets (in the sense of the ambient linear subspace $H_m$). 
To this end, define
\begin{multline*}
D_m^{(i)}=\{(x_1,\dots,x_{m})\in\R^{m}:x_{i+1}\le 0, x_{i+1}+x_{i+2}\le 0,
\dots,\\
x_{i+1}+\ldots+x_{m}+x_1+\ldots+x_{i-1}\le 0,
x_{i+1}+\ldots+x_{m}+x_1+\ldots+x_{i}= 0\},
\end{multline*}
for $i\in\{0,\dots,m-1\}$, where obviously $D_m^{(0)}=D_m$. Observe  that $D_m^{(i)}=\mathcal O D_m$ for a suitable permutation matrix $\mathcal O \in\R^{m\times m}$, and thus, the rotation invariance of the spherical Lebesgue measure yields that the solid angle is the same for each $D_m^{(i)}$, $0\le i\le m$. Furthermore, it is easy to see that $D_m^{(i)}$ can be written as follows:
\begin{align*}
D_m^{(i)}=\big\{x\in\R^m:x_1+\ldots+x_m=0,x_1+\ldots+x_i\ge \max\{x_1,x_1+x_2,\ldots,x_1+\ldots+x_m\}\big\}.
\end{align*}
From this representation of $D_m^{(i)}$ we easily observe that for any point $x\in H_m$, there is a suitable index $i\in\{0,\dots,m-1\}$, namely the one where $x_1+\ldots + x_i$ is  maximal, such that $x\in D_m^{(i)}$. On the other hand, if $x$ is in the intersection of two cones $D_m^{(i)}$ and $D_m^{(j)}$, $i\neq j$, then $x$ belongs to $\{x\in\R^m:x_1+\ldots+x_m=0,x_1+\ldots+x_i=x_1+\ldots+x_j\}$, which is a subspace of dimension $m-2$, and thus, a Lebesgue null set (in the sense of the ambient subspace $H_m$). This yields that $\alpha(D_m)=\frac 1m$.
Note that a probabilistic version of this reasoning is due to Sparre Andersen~\cite{sparre_andersen1} who proved that the probability that a random bridge (satisfying certain minor conditions) of length $m$ stays below zero is $1/m$.
\end{proof}

\begin{proof}[Proof of Theorem~\ref{Theorem_Intr_Vol_An-1}]
In order to prove Theorem~\ref{Theorem_Intr_Vol_An-1}, we want to evaluate the internal and external angles of the faces of $C^A$. Observe that the faces of $C^A$ can be obtained by replacing some of the inequalities defining $C^A$ by the corresponding equalities. It follows that for $2\le k\le n$, any $k$-face of $C^A$ is determined by some collection of indices $1\le l_1<\ldots<l_{k-1}<n$ and given by
\begin{align*}
C^A(l_1,\dots,l_{k-1})=\{\beta\in\R^n: \beta_1=\ldots=\beta_{l_1}\ge \beta_{l_1+1}=\ldots=\beta_{l_2}\ge\ldots\ge\beta_{l_{k-1}+1}=\ldots=\beta_n\}.
\end{align*}
Then, \eqref{Eq_In_Vol_InternalExternal} yields the formula
\begin{align*}
\upsilon_k(C^A)
&	=\sum_{F\in\F_k(C^A)}\alpha(F)\alpha(N(F,C^A))\\
&	=\sum_{1\le l_1<\ldots<l_{k-1}\le n-1}\alpha\big(C^A(l_1,\dots,l_{k-1})\big)\alpha\big(N(C^A(l_1,\dots,l_{k-1}),C^A)\big).
\end{align*}
Note that $C^A$ has no vertices and the only $1$-face is $\{\beta\in\R^n:\beta_1=\ldots=\beta_n\}$, which corresponds to an empty collection of $l$'s.

\vspace*{2mm}
\noindent
\textbf{External Angles.}
Let $1\le k\le n$. Take some $1\le l_1<\ldots<l_{k-1}<n$. We start by computing the normal face $N(C^A(l_1,\dots,l_{k-1}),C^A)$. Our method of proof partially relies on the approach of Gao and Vitale~\cite{Gao2001}. By definition, we have
\begin{align*}
N(C^A(l_1,\dots,l_{k-1}),C^A)=(\lin C^A(l_1,\dots,l_{k-1}))^\perp\cap (C^A)^\circ.
\end{align*}
Using~\eqref{Eq_Duality_Pos_Cap}, we obtain
\begin{align*}
(C^A)^\circ
&	=\{x\in\R^n:x_1\ge x_2\ge\ldots\ge x_n\}^\circ\\
&	=\pos\{e_1,e_1+e_2,\dots,e_1+\ldots+e_n,-e_1-\ldots-e_n\}^\circ\\
&	=\{x\in\R^n:x_1\le 0,x_1+x_2\le 0,\dots,x_1+\ldots+x_{n-1}\le 0,x_1+\ldots+x_n=0\}.
\end{align*}
Here, $e_1,\dots,e_n$ denotes the standard Euclidean orthonormal basis of $\R^n$.
Furthermore,
\begin{align*}
&(\lin C^A(l_1,\dots,l_{k-1}))^\perp\\
&	\quad=\{x\in\R^n:x_1=\ldots=x_{l_1},x_{l_1+1}=\ldots=x_{l_2},\dots,x_{l_{k-1}+1}=\ldots=x_n\}^\perp\\
&	\quad=\{x\in\R^n:x_1+\ldots+x_{l_1}=0,x_{l_1+1}+\ldots+x_{l_2}=0,\dots,x_{l_{k-1}+1}+\ldots+x_n=0\}.
\end{align*}
The explicit representations of $(C^A)^\circ$ and $(\lin C^A(l_1,\dots,l_{k-1}))^\perp$ imply the following representation of the normal face as the orthogonal product of $k$ cones:
\begin{align*}
N(C^A(l_1,\dots,l_{k-1}),C^A)=D_{l_1}\times D_{l_2-l_1}\times\ldots\times D_{l_{k-1}-l_{k-2}}\times D_{n-l_{k-1}},
\end{align*}
where
\begin{align*}
D_m
	:=&\{x\in\R^m:x_1\le 0,x_1+x_2\le 0,\dots,x_1+\ldots+x_{m-1}\le 0,x_1+\ldots+x_m=0\}.
\end{align*}
In Lemma~\ref{lemma:angle_Dm} we proved that $\alpha (D_m) = 1/m$ holds for any $m\in\N$. Since the normal cone $N(C^A(l_1,\dots,l_{k-1}),C^A)$ is the orthogonal product of the $k$ cones $D_{l_1},D_{l_2-l_1},\dots,D_{n-l_{k-1}}$, it follows that
\begin{align}\label{Eq_External_Angle_An-1}
\alpha\big(N(C^A(l_1,\dots,l_{k-1}),C^A)\big)
	=\frac{1}{l_1(l_2-l_1)\cdot\ldots\cdot(n-l_{k-1})}.
\end{align}

\vspace*{2mm}
\noindent
\textbf{Internal Angles.}
Now, we need to compute the internal angles at the faces of $C^A$.  To this end, we consider the linear hull of $C^A(l_1,\dots,l_{k-1})$ which forms a $k$-dimensional linear subspace of $\R^n$. Recall that for $1\le l_1<\ldots<l_{k-1}\le n-1$ the linear hull of $C^A(l_1,\dots,l_{k-1})$ is given by
\begin{align*}
\lin C^A(l_1,\dots,l_{k-1})=\{\beta\in\R^n:\beta_1=\ldots=\beta_{l_1},\dots,\beta_{l_{k-1}+1}=\ldots=\beta_n\}.
\end{align*}
Thus, the vectors $y_i$ given by
\begin{multline*}
y_1=\frac{1}{\sqrt{l_1}}(\overbrace{1,\dots,1}^{1,\dots,l_1},0,\dots,0),
y_2=\frac{1}{\sqrt{l_2-l_1}}(0,\dots,0,\overbrace{1,\dots,1}^{l_1+1,\dots,l_2},0,\dots,0),\dots,\\ y_{k-1}=\frac{1}{\sqrt{l_{k-1}-l_{k-2}}}(0,\dots,0,\overbrace{1,\dots,1}^{l_{k-2}+1,\dots,l_{k-1}},0,\dots,0),y_{k}=\frac{1}{\sqrt{n-l_{k-1}}}(0,\dots,0,\overbrace{1,\dots,1)}^{l_{k-1}+1,\dots,n}
\end{multline*}
form an orthonormal basis of $\lin C^A(l_1,\dots,l_{k-1})$. For independent and standard Gaussian random variables $\xi_1,\dots,\xi_k$, the random vector $N:=\xi_1y_1+\dots+\xi_ky_k$ is $k$-dimensional standard Gaussian on the linear subspace $\lin C^A(l_1,\dots,l_{k-1})$. Thus, we obtain for the angle of $C^A(l_1,\dots,l_{k-1})$ the following formula:
\begin{align}\label{Eq_Internal_Angle_An-1}
\alpha(C^A(l_1,\dots,l_{k-1}))
&	=\P(N\in C^A(l_1,\dots,l_{k-1}))\notag\\
&	=\P\Big(\frac{\xi_1}{\sqrt{l_1}}\ge \frac{\xi_2}{\sqrt{l_2-l_1}}\ge\ldots\ge \frac{\xi_{k-1}}{\sqrt{l_{k-1}-l_{k-2}}}\ge \frac{\xi_k}{\sqrt{n-l_{k-1}}}\Big).
\end{align}

\vspace*{2mm}
\noindent
\textbf{Conic intrinsic volumes.}
Let $1\le k\le n$. Using the formulas~\eqref{Eq_In_Vol_InternalExternal},~\eqref{Eq_External_Angle_An-1} and~\eqref{Eq_Internal_Angle_An-1} we obtain
\begin{align*}
\upsilon_k(C^A)
&	=\sum_{1\le l_1<\ldots<l_{k-1}\le n-1}\alpha(C^A(l_1,\dots,l_{k-1}))\alpha(N(C^A(l_1,\dots,l_{k-1}),C^A))\\
&	=\sum_{1\le l_1<\ldots<l_{k-1}\le n-1}\frac{\P\Big(\frac{\xi_1}{\sqrt{l_1}}\ge \frac{\xi_2}{\sqrt{l_2-l_1}}\ge\ldots\ge \frac{\xi_{k-1}}{\sqrt{l_{k-1}-l_{k-2}}}\ge \frac{\xi_k}{\sqrt{n-l_{k-1}}}\Big)}{l_1(l_2-l_1)\cdot\ldots\cdot(l_{k-1}-l_{k-2})(n-l_{k-1})}.
\end{align*}
Defining $i_1=l_1,i_2=l_2-l_1,\dots,i_{k-1}=l_{k-1}-l_{k-2},i_k=n-l_{k-1}$, we can change the summation indices in the above sum and obtain
\begin{align}\label{Eq_Sum_1}
\upsilon_k(C^A)=\sum_{\substack{i_1,\dots,i_k\in\N\\i_1+\ldots+i_k= n}}\frac{1}{i_1i_2\cdot\ldots\cdot i_k}\P\Big(\frac{\xi_1}{\sqrt{i_1}}\ge \frac{\xi_2}{\sqrt{i_2}}\ge\ldots\ge \frac{\xi_{k-1}}{\sqrt{i_{k-1}}}\ge \frac{\xi_k}{\sqrt{i_k}}\Big).
\end{align}

Now, fix a permutation $\pi$ from the symmetric group $\mathcal{S}_k$. For each tuple $(i_1,\dots,i_k)\in\N^k$ satisfying $i_1+\ldots+i_k=n$, the tuple $(i_{\pi(1)},\dots,i_{\pi(k)})$ also satisfies $i_{\pi(1)}+\ldots +i_{\pi(k)}=n$. Thus, the sum in \eqref{Eq_Sum_1} does not change if we replace the tuple $(i_1,\dots,i_k)$ by $(i_{\pi(1)},\dots,i_{\pi(k)})$ inside the sum. It follows that
\begin{align*}
\upsilon_k(C^A)
&	=\frac{1}{k!}\sum_{\pi\in\mathcal{S}_k}\:\sum_{\substack{i_1,\dots,i_k\in\N\\i_1+\ldots+i_k= n}}\frac{1}{i_{\pi(1)}i_{\pi(2)}\cdot\ldots\cdot i_{\pi(k)}}\P\Big(\frac{\xi_1}{\sqrt{i_{\pi(1)}}}\ge \frac{\xi_2}{\sqrt{i_{\pi(2)}}}\ge\ldots\ge \frac{\xi_k}{\sqrt{i_{\pi(k)}}}\Big)\\
	&=\frac{1}{k!}\sum_{\substack{i_1,\dots,i_k\in\N\\i_1+\ldots+i_k= n}}\frac{1}{i_1i_2\cdot\ldots\cdot i_k}\sum_{\pi\in\mathcal{S}_k}\P\Big(\frac{\xi_1}{\sqrt{i_{\pi(1)}}}\ge \frac{\xi_2}{\sqrt{i_{\pi(2)}}}\ge\ldots\ge \frac{\xi_k}{\sqrt{i_{\pi(k)}}}\Big)\notag\\
	&=\frac{1}{k!}\sum_{\substack{i_1,\dots,i_k\in\N\\i_1+\ldots+i_k= n}}\frac{1}{i_1i_2\cdot\ldots\cdot i_k}\sum_{\pi\in\mathcal{S}_k}  \P\Big(\frac{\xi_{\pi(1)}}{\sqrt{i_{\pi(1)}}}\ge \frac{\xi_{\pi(2)}}{\sqrt{i_{\pi(2)}}}\ge\ldots\ge \frac{\xi_{\pi(k)}}{\sqrt{i_{\pi(k)}}}\Big)   \notag\\
&	=\frac{1}{k!}\sum_{\substack{i_1,\dots,i_k\in\N\\i_1+\ldots+i_k= n}}\frac{1}{i_1i_2\cdot\ldots\cdot i_k},
\end{align*}
where we used that $\xi_1,\dots,\xi_k$ are independent  and identically distributed, hence  exchangeable. Using the representation \eqref{Eq_Stirling1_AsComp} of the Stirling numbers of the first kind, we obtain
\begin{align*}
\upsilon_k(C^A)=\frac{1}{k!}\stirling{n}{k}\frac{k!}{n!}=\stirling{n}{k}\frac{1}{n!},
\end{align*}
which completes the proof.
\end{proof}
\subsection{Type \texorpdfstring{$\boldsymbol{B_n}$}{B\_n}}
Now, we proceed with the $B_n$-case. Recall the definition of the fundamental Weyl chamber of type $B_n$:
\begin{align*}
C^B:=\mathcal{C}(B_n)=\{\beta\in\R^n:\beta_1\ge\beta_2\ge\ldots\ge\beta_n\ge 0\}.
\end{align*}
Then the analogue of Theorem~\ref{Theorem_Intr_Vol_An-1} is the following.

\begin{satz}\label{Theorem_Intr_Vol_Bn}
For $k=0,1,\dots,n$ the $k$-th conic intrinsic volume of the Weyl chamber $\mathcal{C}(B_n)$ is given by
\begin{align*}
\upsilon_k(\mathcal{C}(B_n))=\frac{\stirlingb{n}{k}}{2^nn!}.
\end{align*}
\end{satz}

Similar to the $A_{n-1}$-case, we first need a formula for the solid angle of a cone whose points correspond to ``walks staying below zero''.

\begin{lem}\label{lemma:angle_Bm}
The solid angle of the cone
$$
B_{m}
:=\{x\in\R^m:x_1\le 0,x_1+x_2\le 0,\dots,x_1+\ldots+x_{m}\le 0\},
$$
is given by $\alpha(B_m)=\binom{2m}m \frac{1}{4^m}$, for $m\in\N$.
\end{lem}

\begin{proof}[Proof of Lemma~\ref{lemma:angle_Bm}]
Take an $m$-dimensional standard Gaussian vector $N=(N_1,\dots,N_m)$. The solid angle of $B_m$ coincides with the probability that $N\in B_m$. Using the well-known formula of Sparre Andersen~\cite{Andersen1949} (which is valid in the special case when  $N_1,\dots,N_{m}$ are independent and standard Gaussian) we obtain
\begin{align*}
\alpha(B_m)=\P[N\in B_m]=\P[N_1\le 0,N_1+N_2\le 0,\dots,N_1+\ldots+N_m\le 0]=\binom{2m}m \frac{1}{2^{2m}},
\end{align*}
which completes the proof.
\end{proof}

\begin{proof}[Proof of Theorem~\ref{Theorem_Intr_Vol_Bn}]
To obtain the faces of $C^B$ we have to replace some of the inequalities defining $C^B$ by the corresponding equalities. It follows that for $1\le k \le n$ any $k$-face of $C^B$ is determined by a collection of indices $1\le l_1<\ldots<l_k\le n$ and given by
\begin{align*}
C^B(l_1,\dots,l_{k})=\{\beta\in\R^n: \beta_1=\ldots=\beta_{l_1}\ge\ldots\ge \beta_{l_{k-1}+1}=\ldots=\beta_{l_k}\ge\beta_{l_{k}+1}=\ldots=\beta_n=0\}.
\end{align*}
Then, \eqref{Eq_In_Vol_InternalExternal} yields the formula
\begin{align*}
\upsilon_k(C^B)
&	=\sum_{1\le l_1<\ldots<l_{k}\le n}\alpha\big(C^B(l_1,\dots,l_{k})\big)\alpha\big(N(C^B(l_1,\dots,l_{k}),C^B)\big).
\end{align*}

\vspace*{2mm}
\noindent
\textbf{External angles.}
The computation of the external angles for the $B_n$-case is similar to the $A_{n-1}$-case. Take some $k\in\{1,\dots,n\}$ and $1\le l_1<  \ldots<l_k\le n$. We start by computing the normal face $N(C^B(l_1,\dots,l_k),C^B)$ of $C^B(l_1,\dots,l_k)$. First of all, we have
\begin{align*}
(C^B)^\circ
&	=\pos\{e_1,e_1+e_2,\dots,e_1+\ldots+e_n\}^\circ\\
&	=\{x\in\R^n:x_1\le 0,x_1+x_2\le 0,\dots,x_1+\ldots+x_n\le 0\}.
\end{align*}
Since
\begin{align*}
(\lin C^B(l_1,\dots,l_k))^\perp
&	=\{x\in\R^n:x_1=\ldots=x_{l_1},\dots,x_{l_{k-1}+1}=\ldots=x_{l_k},x_{l_k+1}=\ldots=x_n=0\}^\perp\\
&	=\{x\in\R^n:x_1+\ldots+x_{l_1}=0,\dots,x_{l_{k-1}+1}+\ldots+x_{l_k}=0\},
\end{align*}
it follows that the normal cone $N(C^B(l_1,\dots,l_k),C^B)$ can be written
as the orthogonal product
\begin{align*}
N(C^B(l_1,\dots,l_k),C^B)=D_{l_1}\times D_{l_2-l_1}\times\ldots\times D_{l_k-l_{k-1}}\times B_{n-l_k},
\end{align*}
where 
\begin{align*}
D_m
&=\{x\in\R^m:x_1\le 0,x_1+x_2\le 0,\dots,x_1+\ldots+x_{m-1}\le 0,x_1+\ldots+x_m=0\},\\
B_{m}
&=\{x\in\R^m:x_1\le 0,x_1+x_2\le 0,\dots,x_1+\ldots+x_{m}\le 0\},\qquad B_0:=\{0\}.
\end{align*}
Using Lemmas~\ref{lemma:angle_Dm} and~\ref{lemma:angle_Bm} and the orthogonal product structure of the normal face
we arrive at
\begin{align}\label{Eq_External_Angle_Bn}
\alpha\big(N(C^B(l_1,\dots,l_k),C^B)\big)=\frac{\dbinom{2(n-l_k)}{n-l_k}}{l_1(l_2-l_1)\cdot\ldots\cdot(l_k-l_{k-1})2^{2(n-l_k)}}.
\end{align}
Note that the above formula is also valid for the case $k=0$ where $C^B(l_1,\dots,l_k)=\{0\}$.
%

\vspace*{2mm}
\noindent
\textbf{Internal angles.}
Now, we need to compute the internal angles for the faces of $C^B$, that is, the solid angles $\alpha(C^B(l_1,\dots,l_k))$ for $1\le l_1<\ldots<l_k\le n$. Recall that the linear hull of $C^B(l_1,\dots,l_k)$ is the $k$-dimensional linear subspace given by 
\begin{align*}
\lin C^B(l_1,\dots,l_k)=\{x\in\R^n:x_1=\ldots=x_{l_1},x_{l_1+1}=\ldots=x_{l_2},\dots,x_{l_k+1}=\ldots=x_n=0\}.
\end{align*}
The following vectors $y_1,\dots,y_k$ form an orthonormal basis of $\lin C^B(l_1,\dots,l_k)$:
\begin{multline*}
y_1=\frac{1}{\sqrt{l_1}}(\overbrace{1,\dots,1}^{1,\dots,l_1},0,\dots,0),y_2=\frac{1}{\sqrt{l_2-l_1}}(0,\dots,0,\overbrace{1,\dots,1}^{l_1+1,\dots,l_2},0,\dots,0),\dots,\\ y_k=\frac{1}{\sqrt{l_{k}-l_{k-1}}}(0,\dots,0,\overbrace{1,\dots,1}^{l_{k-1}+1,\dots,l_k},0,\dots,0).
\end{multline*}
For independent and standard Gaussian random variables $\xi_1,\dots,\xi_k$, the random vector $N:=\xi_1y_1+\dots+\xi_ky_k$ is $k$-dimensional standard Gaussian on the linear subspace $\lin C^B(l_1,\dots,l_k)$. Thus, we obtain for the angle of $C^B(l_1,\dots,l_k)$ the following formula:
\begin{align}\label{Eq_Internal_Angle_Bn}
\alpha(C^B(l_1,\dots,l_k))
	=\P(N\in C^B(l_1,\dots,l_k))=\P\Big(\frac{\xi_1}{\sqrt{l_1}}\ge \frac{\xi_2}{\sqrt{l_2-l_1}}\ge\ldots\ge \frac{\xi_k}{\sqrt{l_k-l_{k-1}}}\ge 0\Big).
\end{align}

\vspace*{2mm}
\noindent
\textbf{Conic intrinsic volumes.}
Let $1\le k\le n$. Using the formulas~\eqref{Eq_Internal_Angle_Bn} and~\eqref{Eq_External_Angle_Bn} for the internal and external angles, we have
\begin{align*}
\upsilon_k(C^B)
&	=\sum_{1\le l_1<\ldots<l_k\le n}\alpha(C^B(l_1,\dots,l_k))\alpha(N(C^B(l_1,\dots,l_k),C^B))\\
&	=\sum_{1\le l_1<\ldots<l_k\le n}\dbinom{2(n-l_k)}{n-l_k}\frac{\P\Big(\frac{\xi_1}{\sqrt{l_1}}\ge \frac{\xi_2}{\sqrt{l_2-l_1}}\ge\ldots\ge \frac{\xi_k}{\sqrt{l_k-l_{k-1}}}\ge 0\Big)}{l_1(l_2-l_1)\cdot\ldots\cdot(l_k-l_{k-1})2^{2(n-l_k)}}.
\end{align*}
Defining $i_1=l_1,i_2=l_2-l_1,\dots,i_k=l_k-l_{k-1}$, we can change the summation indices in the above sum and obtain
\begin{align}\label{Eq_Sum_Bn_1}
\upsilon_k(C^B)
&=\sum_{\substack{i_1,\dots,i_k\in\N\\i_1+\ldots+i_k\le n}}\dbinom{2(n-i_1-\ldots-i_k)}{n-i_1-\ldots-i_k}\frac{\P\Big(\frac{\xi_1}{\sqrt{i_1}}\ge \frac{\xi_2}{\sqrt{i_2}}\ge\ldots\ge \frac{\xi_k}{\sqrt{i_k}}\ge 0\Big)}{i_1i_2\cdot\ldots\cdot i_k 2^{2(n-i_1-\ldots-i_k)}}\notag\\
&=\sum_{r=0}^{n-k}\sum_{\substack{i_1,\dots,i_k\in\N\\i_1+\ldots+i_k= n-r}}\dbinom{2r}{r}\frac{\P\Big(\frac{\xi_1}{\sqrt{i_1}}\ge \frac{\xi_2}{\sqrt{i_2}}\ge\ldots\ge \frac{\xi_k}{\sqrt{i_k}}\ge 0\Big)}{i_1i_2\cdot\ldots\cdot i_k 2^{2r}}.
\end{align}
Now, fix a number $r\in\{0,1,\dots,n-k\}$, a permutation $\pi\in\mathcal{S}_k$ and a vector of signs $\eps=(\eps_1,\dots,\eps_k)\in\{\pm 1\}^k$. For each tuple $(i_1,\dots,i_k)\in\N^k$ satisfying $i_1+\ldots+i_k=n-r$, the tuple $(i_{\pi(1)},\dots,i_{\pi(k)})$ also satisfies $i_{\pi(1)}+\ldots +i_{\pi(k)}=n-r$. Thus, the inner sum in~\eqref{Eq_Sum_Bn_1} does not change if we replace the tuple $(i_1,\dots,i_k)$ by $(i_{\pi(1)},\dots,i_{\pi(k)})$ inside the sum. Furthermore, the sum does not change if we additionally replace $\xi_1,\dots,\xi_k$ by $\eps_1\xi_1,\dots,\eps_k\xi_k$, since $\xi_1,\dots,\xi_k$ are independent and  standard normal. Thus, we obtain
\begin{align*}
\upsilon_k(C^B)
&	=\sum_{r=0}^{n-k}\frac{1}{2^kk!}\sum_{(\eps,\pi)\in\{\pm 1\}^k\times \mathcal{S}_k}\:\sum_{\substack{i_1,\dots,i_k\in\N\\i_1+\ldots+i_k= n-r}}\dbinom{2r}{r}\frac{\P\Big(\frac{\eps_1\xi_1}{\sqrt{i_{\pi(1)}}}\ge \frac{\eps_2\xi_2}{\sqrt{i_{\pi(2)}}}\ge\ldots\ge \frac{\eps_k\xi_k}{\sqrt{i_{\pi(k)}}}\ge 0\Big)}{i_{\pi(1)}i_{\pi(2)}\cdot\ldots\cdot i_{\pi(k)} 2^{2r}}\\
&	=\frac{1}{2^kk!}\sum_{r=0}^{n-k}\sum_{\substack{i_1,\dots,i_k\in\N\\i_1+\ldots+i_k= n-r}}\frac{\binom{2r}{r}}{i_1i_2\cdot\ldots\cdot i_k2^{2r}}\sum_{(\eps,\pi)\in\{\pm 1\}^k\times \mathcal{S}_k}\P\Big(\frac{\eps_1\xi_1}{\sqrt{i_{\pi(1)}}}\ge\ldots\ge \frac{\eps_k\xi_k}{\sqrt{i_{\pi(k)}}}\ge 0\Big)\\
&	=\frac{1}{2^kk!}\sum_{r=0}^{n-k}\sum_{\substack{i_1,\dots,i_k\in\N\\i_1+\ldots+i_k= n-r}}\frac{\binom{2r}{r}}{i_1i_2\cdot\ldots\cdot i_k2^{2r}}\sum_{(\eps,\pi)\in\{\pm 1\}^k\times \mathcal{S}_k}\P\Big(\frac{\eps_1\xi_{\pi(1)}}{\sqrt{i_{\pi(1)}}}\ge\ldots\ge \frac{\eps_k\xi_{\pi(k)}}{\sqrt{i_{\pi(k)}}}\ge 0\Big)\\
&	=\frac{1}{2^kk!}\sum_{r=0}^{n-k}\sum_{\substack{i_1,\dots,i_k\in\N\\i_1+\ldots+i_k= n-r}}\frac{\binom{2r}{r}}{i_1i_2\cdot\ldots\cdot i_k2^{2r}}.
\end{align*}
Using the representation~\eqref{Eq_Stirling1_AsComp} for the Stirling numbers of the first kind and the formula~\eqref{Eq_Equality_Generating_functions}, we obtain
\begin{align*}
\upsilon_k(C^B)
=	\sum_{r=0}^{n-k}\frac{1}{2^kk!}\frac{\binom{2r}{r}}{2^{2r}}\stirling{n-r}{k}\frac{k!}{(n-r)!}
=	\sum_{r=0}^{n-k}2^{-k-2r}\binom{2r}{r}\stirling{n-r}{k}\frac{1}{(n-r)!}=\frac{\stirlingb nk}{2^nn!}.
\end{align*}
The case $k=0$ is easy since the only $0$-face is the origin $\{0\}$, and thus
\begin{align*}
\upsilon_0(C^B)=\alpha(\{0\})\alpha(N(\{0\},C^B))=\alpha((C^B)^\circ)=\alpha(B_n)
=\binom{2n}n\frac{1}{2^{2n}}=\frac{(2n-1)!!}{2^nn!}=\frac{\stirlingb n0}{2^nn!},
\end{align*}
where we applied the formula for $\alpha(B_n)$ stated in Lemma~\ref{lemma:angle_Bm}.
\end{proof}

\subsection{Type \texorpdfstring{$\boldsymbol{D_n}$}{D\_n}}

At last, we can also compute the conic intrinsic volumes of the Weyl chambers of type $D_n$. They follow from the intrinsic volumes of the type $B_n$ chambers and the additivity of the intrinsic volumes.

\begin{satz}\label{Theorem_Intr_Vol_Dn}
For $k=0,1,\dots,n$ the $k$-th conic intrinsic volume of the Weyl chamber of type $D_n$ is given by
\begin{align*}
\upsilon_k(\mathcal{C}(D_n))=\frac{\stirlingd {n}{k}}{2^{n-1}n!},
\end{align*}
where the numbers $\stirlingd {n}{k}$ are given by~\eqref{eq:D_n_k}.
\end{satz}

\begin{proof}
Consider the Weyl chamber
\begin{align*}
\mathcal{C}(D_n)
&   =\{\beta\in\R^n:\beta_1\ge\ldots\ge \beta_{n-1} \ge |\beta_n|\}\\
&	=\big(\mathcal{C}(D_n)\cap\{\beta_n\ge 0\}\big)\cup\big(\mathcal{C}(D_n)\cap\{\beta_n\le 0\}\big)\\
&	=\{\beta\in\R^n:\beta_1\ge \ldots\ge\beta_{n-1} \ge\beta_n\ge 0\}\cup\{\beta\in\R^n:\beta_1\ge \ldots\ge \beta_{n-1}\ge -\beta_n\ge 0\},
\end{align*}
where the first set on the right-hand side is the Weyl chamber of type $B_n$, while the second set is isometric to it.  Their intersection
\begin{align*}
&\{\beta\in\R^n:\beta_1\ge \ldots\ge\beta_{n-1} \ge\beta_n\ge 0\}\cap\{\beta\in\R^n:\beta_1\ge \ldots\ge \beta_{n-1}\ge -\beta_n\ge 0\}\\
&	\quad=\{\beta_1\ge\ldots\ge\beta_{n-1}\ge\beta_{n}=0\}
\end{align*}
is a Weyl chamber of type $B_{n-1}$ if we identify $\R^{n-1}$ and $\R^{n-1}\times \{0\}$. The conic intrinsic volumes $\upsilon_k$ are additive functionals, see \cite[Theorem 6.5.2]{Schneider2008},  and thus, we obtain
\begin{align*}
\upsilon_k(\mathcal{C}(D_n))
&	= 2\upsilon_k(\{\beta\in\R^n:\beta_1\ge \ldots\ge\beta_{n-1} \ge\beta_n\ge 0\})
-\upsilon_k(\{\beta_1\ge\ldots\ge\beta_{n-1}\ge\beta_{n}=0\})\\
&	= 2\frac{\stirlingb{n}{k}}{2^nn!}-\frac{\stirlingb{n-1}{k}}{2^{n-1}(n-1)!}	=\frac{\stirlingb{n}{k}-n\stirlingb{n-1}{k}}{2^{n-1}n!}
=\frac{\stirlingd {n}{k}}{2^{n-1}n!}.
\end{align*}
Here, we used the formula for the intrinsic volumes of the Weyl chambers of type $B_n$ from Theorem~\ref{Theorem_Intr_Vol_Bn}, and in the last step the relation~\eqref{Eq_Relation_D(n,k)_B(n,k)}.
\end{proof}

\section*{Acknowledgement}
Supported by the German Research Foundation under Germany's Excellence Strategy  EXC 2044 -- 390685587, \textit{Mathematics M\"unster: Dynamics - Geometry - Structure}  and by the DFG priority program SPP 2265 \textit{Random Geometric Systems}.

\vspace{1cm}

\bibliography{bibliography}

\begin{thebibliography}{10}

\bibitem{Abramson+Pitman:2011}
J.~Abramson and J.~Pitman.
\newblock Concave majorants of random walks and related {P}oisson processes.
\newblock {\em Combin. Probab. Comput.}, 20(5):651--682, 2011.

\bibitem{Abramson+Pitman+Ross+Bravo:2011}
J.~Abramson, J.~Pitman, N.~Ross, and G.~Uribe~Bravo.
\newblock Convex minorants of random walks and {L}\'{e}vy processes.
\newblock {\em Electron. Commun. Probab.}, 16:423--434, 2011.

\bibitem{amelunxen_buergisser}
D.~Amelunxen and P.~B\"{u}rgisser.
\newblock Intrinsic volumes of symmetric cones and applications in convex
  programming.
\newblock {\em Math. Program.}, 149(1-2, Ser. A):105--130, 2015.

\bibitem{amelunxen_buergisser_grass}
D.~Amelunxen and P.~B\"{u}rgisser.
\newblock Probabilistic analysis of the {G}rassmann condition number.
\newblock {\em Found. Comput. Math.}, 15(1):3--51, 2015.

\bibitem{Amelunxen2017}
D.~Amelunxen and M.~Lotz.
\newblock Intrinsic volumes of polyhedral cones: A combinatorial perspective.
\newblock {\em Discrete Comput. Geom.}, 58(2):371--409, jul 2017.

\bibitem{AmelunxenLotzDCG17}
D.~Amelunxen and M.~Lotz.
\newblock Intrinsic volumes of polyhedral cones: a combinatorial perspective.
\newblock {\em Discrete Comput. Geom.}, 58(2):371--409, 2017.

\bibitem{Amelunxen2014}
D.~Amelunxen, M.~Lotz, M.~B. McCoy, and J.~A. Tropp.
\newblock Living on the edge: phase transitions in convex programs with random
  data.
\newblock {\em Information and Inference}, 3(3):224--294, jun 2014.

\bibitem{bagno_biagioli_garber_some_identities}
E.~Bagno, R.~Biagioli, and D.~Garber.
\newblock Some identities involving second kind {S}tirling numbers of types {B}
  and {D}.
\newblock {\em Elect. J. Combin.}, 26(3):P3.9, 2019.

\bibitem{bagno_garber_balls}
E.~Bagno and D.~Garber.
\newblock Signed partitions - {A} balls into urns approach, 2019.
\newblock Preprint at arXiv: 1903.02877.

\bibitem{bala_stirling}
P.~Bala.
\newblock A {$3$}-parameter family of generalized {S}tirling numbers, 2015.
\newblock Preprint at https://oeis.org/A143395/a143395.pdf.

\bibitem{bona_handbook}
M.~B\'{o}na, editor.
\newblock {\em Handbook of enumerative combinatorics}.
\newblock Discrete Mathematics and its Applications (Boca Raton). CRC Press,
  Boca Raton, FL, 2015.

\bibitem{dowling}
T.~A. Dowling.
\newblock A class of geometric lattices based on finite groups.
\newblock {\em J. Combinatorial Theory Ser. B}, 14:61--86, 1973.

\bibitem{drton_klivans}
M.~{Drton} and C.~J. {Klivans}.
\newblock {A geometric interpretation of the characteristic polynomial of
  reflection arrangements.}
\newblock {\em {Proc. Amer. Math. Soc.}}, 138(8):2873--2887, 2010.

\bibitem{gao}
F.~Gao.
\newblock The mean of a maximum likelihood estimator associated with the
  {B}rownian bridge.
\newblock {\em Electron. Comm. Probab.}, 8:1--5, 2003.

\bibitem{Gao2001}
F.~Gao and R.~A. Vitale.
\newblock Intrinsic volumes of the {B}rownian motion body.
\newblock {\em Discrete Comput. Geom.}, 26(1):41--50, jan 2001.

\bibitem{henze_orakel}
N.~Henze.
\newblock Weitere {{\"U}}berraschungen im {Z}usammenhang mit dem
  {S}chnur-{O}rakel.
\newblock {\em Stochastik in der Schule}, 33(3):18--23, 2013.

\bibitem{HugSchneider2016}
D.~Hug and R.~Schneider.
\newblock Random conical tessellations.
\newblock {\em Discrete Comput. Geom.}, 56(2):395--426, may 2016.

\bibitem{humphreys_book}
J.~E. Humphreys.
\newblock {\em Reflection groups and {C}oxeter groups}, volume~29 of {\em
  Cambridge Studies in Advanced Mathematics}.
\newblock Cambridge University Press, Cambridge, 1990.

\bibitem{KVZ17}
Z.~Kabluchko, V.~Vysotsky, and D.~Zaporozhets.
\newblock Convex hulls of random walks: expected number of faces and face
  probabilities.
\newblock {\em Adv. Math.}, 320:595--629, 2017.

\bibitem{KVZ15}
Z.~Kabluchko, V.~Vysotsky, and D.~Zaporozhets.
\newblock Convex hulls of random walks, hyperplane arrangements, and {W}eyl
  chambers.
\newblock {\em Geom. Funct. Anal.}, 27(4):880--918, 2017.

\bibitem{klivans_swartz}
C.~J. {Klivans} and E.~{Swartz}.
\newblock {Projection volumes of hyperplane arrangements.}
\newblock {\em {Discrete Comput. Geom.}}, 46(3):417--426, 2011.

\bibitem{lang_stirling}
W.~Lang.
\newblock On sums of powers of arithmetic progressions, and generalized
  {S}tirling, {E}ulerian and {B}ernoulli numbers, 2017.
\newblock Preprint at arXiv: 1707.04451.

\bibitem{Pitman2006}
J.~Pitman.
\newblock {\em Combinatorial Stochastic Processes}.
\newblock Lecture Notes in Mathematics. Springer-Verlag, 2006.

\bibitem{schneider_combinatorial}
R.~Schneider.
\newblock Combinatorial identities for polyhedral cones.
\newblock {\em St. Petersburg Math. J.}, 29(1):209--221, 2018.

\bibitem{Schneider2008}
R.~Schneider and W.~Weil.
\newblock {\em Stochastic and Integral Geometry}.
\newblock Springer Berlin Heidelberg, 2008.

\bibitem{shephard_todd}
G.~C. Shephard and J.~A. Todd.
\newblock Finite unitary reflection groups.
\newblock {\em Canad. J. Math.}, 6:274--304, 1954.

\bibitem{sloane}
N.~J.~A. Sloane~(editor).
\newblock {The {O}n-{L}ine {E}ncyclopedia of {I}nteger {S}equences}.
\newblock https://oeis.org.

\bibitem{Andersen1949}
E.~{Sparre Andersen}.
\newblock On the number of positive sums of random variables.
\newblock {\em Scandinavian Actuarial Journal}, 1949(1):27--36, jan 1949.

\bibitem{sparre_andersen1}
E.~{Sparre Andersen}.
\newblock {On the fluctuations of sums of random variables.}
\newblock {\em {Math. Scand.}}, 1:263--285, 1953.

\bibitem{Sparre}
E.~{Sparre Andersen}.
\newblock On the fluctuations of sums of random variables {II}.
\newblock {\em Math. Scand.}, 2:195--223, 1954.

\bibitem{stanley_book}
R.~P. {Stanley}.
\newblock {An introduction to hyperplane arrangements.}
\newblock In E.~Miller, V.~Reiner, and B.~Sturmfels, editors, {\em {Geometric
  combinatorics}}, volume~13 of {\em IAS/Park City Mathematics Series}, pages
  389--496. AMS, 2007.

\bibitem{suter}
R.~Suter.
\newblock Two analogues of a classical sequence.
\newblock {\em J. Integer Seq.}, 3(1):Article 00.1.8, 2000.

\end{thebibliography}
\bibliographystyle{abbrv}



\end{document}